\documentclass[english,oneside,a4paper,10pt]{article}
\usepackage{amssymb}
\usepackage{latexsym}
\usepackage[english]{babel}
\usepackage{amsmath}
\usepackage{amsthm}
\usepackage{amsfonts}
\usepackage{amsbsy}
\usepackage[mathscr]{eucal}
\usepackage[latin1]{inputenc}
\usepackage{verbatim}
\usepackage{mathrsfs}
\usepackage{graphicx}
\usepackage{float}
\usepackage{bbm}
\usepackage{lscape}
\usepackage{mathdots}
\usepackage{empheq}
\usepackage[dvipsnames]{xcolor}
\usepackage{array}
\usepackage{a4wide}
\usepackage{enumitem}
\usepackage[colorlinks,bookmarks,breaklinks]{hyperref}  
\hypersetup{linkcolor=black,citecolor=black,filecolor=black,urlcolor=black}

\providecommand{\abs}[1]{\left\lvert#1 \right\rvert}

\newtheorem{theorem}{Theorem}[section]
\newtheorem{lemma}[theorem]{Lemma}
\newtheorem{corollary}[theorem]{Corollary}

\theoremstyle{remark}
\newtheorem{remark}[theorem]{Remark}

\numberwithin{equation}{section}

\allowdisplaybreaks

\begin{document}
\begin{center}
{\huge {\bf On the convergence of series of moments for row sums of random variables}} \\
\vspace{1.0cm}
{\Large Jo\~{a}o Lita da Silva\footnote{\textit{E-mail address:} \texttt{jfls@fct.unl.pt}; \texttt{joao.lita@gmail.com}}} \\
\vspace{0.1cm}
\textit{Department of Mathematics and GeoBioTec \\ Faculty of Sciences and Technology \\
NOVA University of Lisbon \\ Quinta da Torre, 2829-516 Caparica,
Portugal}
\end{center}

\bigskip

\bigskip

\begin{abstract}
Given a triangular array $\left\{X_{n,k}, \, 1 \leqslant k \leqslant n, n \geqslant 1 \right\}$ of random variables satisfying $\mathbb{E} \lvert X_{n,k} \rvert^{p} < \infty$ for some $p \geqslant 1$ and sequences $\{b_{n} \}$, $\{c_{n} \}$ of positive real numbers, we shall prove that $\sum_{n=1}^{\infty} c_{n} \mathbb{E} \left[\abs{\sum_{k=1}^{n} (X_{n,k} - \mathbb{E} \, X_{n,k})}/b_{n} - \varepsilon \right]_{+}^{p} < \infty$, where $x_{+} = \max(x,0)$. Our results are announced in a general setting, allowing us to obtain the convergence of the series in question under various types of dependence.
\end{abstract}

\bigskip


\bigskip


\bigskip

\section{Introduction}\label{sec:1}

In \cite{Li05}, Li and Sp\u{a}taru proved the following statement: if $\{X_{n}, \, n \geqslant 1 \}$ is a sequence of independent and identically distributed (i.i.d.) random variables with $\mathbb{E} \, X_{1} = 0$ and $p > 0$, $0 < q < 2$, $r \geqslant 1$ are such that $qr \geqslant 1$, then
\begin{equation}\label{eq:1.1}
\left\{
\begin{array}{l}
  \mathbb{E} \lvert X_{1} \rvert^{p} < \infty \; \; \text{if $p > qr$} \\[5pt]
  \mathbb{E} \lvert X_{1} \rvert^{qr} \log(1 + \lvert X_{1} \rvert) < \infty \; \; \text{if $p = qr$} \\[5pt]
  \mathbb{E} \lvert X_{1} \rvert^{qr} < \infty \; \; \text{if $p < qr$}
\end{array}
\right.
\end{equation}
is equivalent to
\begin{equation*}
\int_{\varepsilon}^{\infty} \sum_{n=1}^{\infty} n^{r - 2} \mathbb{P} \left\{\abs{\sum_{k=1}^{n} X_{k}} > x^{1/p} n^{1/q} \right\} \, \mathrm{d}x < \infty \; \; \text{for all $\varepsilon > 0$}.
\end{equation*}
A few years later, Chen and Wang \cite{Chen08} showed that, letting $p>0$, $\{X_{n}, \, n \geqslant 1 \}$ be a random sequence and $\{b_{n} \}$, $\{c_{n} \}$ be sequences of positive real numbers,
\begin{equation*}
\int_{\varepsilon}^{\infty} \sum_{n=1}^{\infty} c_{n} \mathbb{P} \left\{\abs{X_{n}} > x^{1/p} b_{n} \right\} \, \mathrm{d}x < \infty \; \; \text{for all $\varepsilon > 0$}
\end{equation*}
and
\begin{equation*}
\sum_{n=1}^{\infty} c_{n} \, \mathbb{E} \left[\max\left(\frac{\lvert X_{n} \rvert}{b_{n}} - \varepsilon,0 \right) \right]^{p} < \infty \; \; \text{for all $\varepsilon > 0$}
\end{equation*}
are equivalent. Hence, putting $x_{+} = \max(x,0)$, Li and Sp\u{a}taru's result can be restated as: if $\{X_{n}, \, n \geqslant 1 \}$ is a sequence of i.i.d. random variables with $\mathbb{E} \, X_{1} = 0$ and $p > 0$, $0 < q < 2$, $r \geqslant 1$ are such that $qr \geqslant 1$, then \eqref{eq:1.1} is equivalent to
\begin{equation}\label{eq:1.2}
\sum_{n=1}^{\infty} n^{r - 2} \mathbb{E} \left(n^{-1/q} \abs{\sum_{k=1}^{n} X_{k}} - \varepsilon \right)_{+}^{p} < \infty \; \; \text{for all $\varepsilon > 0$}.
\end{equation}
The extension of \eqref{eq:1.2} to arrays of (dependent) random variables has been emerged in literature over the last years (see \cite{Sung09}, \cite{Wu13}, \cite{Wu14}, \cite{Wu15} or more recently \cite{Yi18}). Our purpose in this paper is to give general sufficient conditions to obtain
\begin{equation}\label{eq:1.3}
\sum_{n=1}^{\infty} c_{n} \mathbb{E} \left[\frac{\abs{\sum_{k=1}^{n} (X_{n,k} - \mathbb{E} \, X_{n,k})}}{b_{n}} - \varepsilon \right]_{+}^{p} < \infty \; \; \text{for all $\varepsilon > 0$}
\end{equation}
when $\left\{X_{n,k}, \, 1 \leqslant k \leqslant n, n \geqslant 1 \right\}$ is a triangular array of random variables and $\{b_{n} \}$, $\{c_{n} \}$ are sequences of positive constants. Namely, we shall assume that row sums of a suitable truncated triangular array of random variables satisfies classical moment inequalities, scilicet, a von Bahr-Esseen type inequality \cite{von65} or a Rosenthal type inequality (see, for instance, \cite{Petrov95} page $59$). These are general assumptions which cover well-known dependent structures, particularly extended negatively dependence or pairwise negatively quadrant dependence (see Lemma~\ref{lem:3.1} in last section).

In the sequel, we shall denote the indicator random variable of an event $A$ by $I_{A}$ and, for each $t > 0$, we shall define also the function $g_{t}(x) = \max(\min(x,t),-t)$ which describes the truncation at level $t$.


\section{Main results}\label{sec:2}

In our first two statements, we shall establish the convergence of series \eqref{eq:1.3} by assuming that, for any $t > 0$, the (truncated) array of random variables $\left\{g_{t}(X_{n,k}), \, 1 \leqslant k \leqslant n, n \geqslant 1 \right\}$ satisfies a von Bahr-Esseen type inequality, i.e. there is a sequence of positive numbers $\{\alpha_{n} \}$ such that for some $q > 1$,
\begin{equation}\label{eq:2.1}
\mathbb{E} \abs{\sum_{k=1}^{n} \left[g_{t}(X_{n,k}) - \mathbb{E} \, g_{t}(X_{n,k}) \right]}^{q} \leqslant \alpha_{n} \sum_{k=1}^{n} \mathbb{E} \, \lvert g_{t}(X_{n,k}) \rvert^{q}
\end{equation}
for all $n \geqslant 1$ and $t > 0$.

\begin{theorem}\label{thr:2.1}
Let $p > 1$, $\left\{X_{n,k}, \, 1 \leqslant k \leqslant n, n \geqslant 1 \right\}$ be an array of random variables satisfying $\mathbb{E} \lvert X_{n,k} \rvert^{p} < \infty$ for each $1 \leqslant k \leqslant n$ and all $n \geqslant 1$, and verifying \eqref{eq:2.1} for a $q > p$ and some sequence $\{\alpha_{n} \}$ of positive numbers. If $\{b_{n} \}$, $\{c_{n} \}$ are real sequences of positive numbers such that \\

\noindent \textnormal{(a)} $\sum_{n=1}^{\infty} \sum_{k=1}^{n} \alpha_{n} c_{n} b_{n}^{-q} \int_{0}^{b_{n}^{q}} \mathbb{P} \left\{\lvert X_{n,k} \rvert^{q} > t \right\} \, \mathrm{d}t < \infty$, \\

\noindent \textnormal{(b)} $\sum_{n=1}^{\infty} \sum_{k=1}^{n} c_{n} \mathbb{E} \, \lvert X_{n,k} \rvert I_{\left\{\lvert X_{n,k} \rvert > b_{n} \right\}}/b_{n} < \infty$, \\

\noindent \textnormal{(c)} $\sum_{n=1}^{\infty} \sum_{k=1}^{n} (1 + \alpha_{n}) c_{n} b_{n}^{-p} \int_{b_{n}^{p}}^{\infty} \mathbb{P} \left\{\lvert X_{n,k} \rvert^{p} > t \right\} \, \mathrm{d}t < \infty$, \\

\noindent then
\begin{equation*}
\sum_{n=1}^{\infty} c_{n} \mathbb{E} \left[\frac{\abs{\sum_{k=1}^{n} (X_{n,k} - \mathbb{E} \, X_{n,k})}}{b_{n}} - \varepsilon \right]_{+}^{p} < \infty
\end{equation*}
for all $\varepsilon > 0$.
\end{theorem}

\begin{proof}
Fixing $\varepsilon > 0$, we have
\begin{equation}\label{eq:2.2}
\begin{split}
& \mathbb{E} \left[\abs{\sum_{k=1}^{n} (X_{n,k} - \mathbb{E} \, X_{n,k})} - \varepsilon b_{n} \right]_{+}^{p} \\
& \qquad = \int_{0}^{\infty} \mathbb{P} \left\{\abs{\sum_{k=1}^{n} (X_{n,k} - \mathbb{E} \, X_{n,k})} > \varepsilon b_{n} + t^{1/p} \right\} \mathrm{d}t \\
& \qquad \leqslant b_{n}^{p} \mathbb{P} \left\{\abs{\sum_{k=1}^{n} (X_{n,k} - \mathbb{E} \, X_{n,k})} > \varepsilon b_{n} \right\} + \int_{b_{n}^{p}}^{\infty} \mathbb{P} \left\{\abs{\sum_{k=1}^{n} (X_{n,k} - \mathbb{E} \, X_{n,k})} > t^{1/p} \right\} \mathrm{d}t.
\end{split}
\end{equation}
Defining $X_{n,k}' := g_{b_{n}}(X_{n,k})$ and $X_{n,k}'' = X_{n,k} - X_{n,k}'$, Chebyshev inequality and assumption \eqref{eq:2.1} entail
\begin{equation}\label{eq:2.3}
\begin{split}
& \mathbb{P} \left\{\abs{\sum_{k=1}^{n} (X_{n,k} - \mathbb{E} \, X_{n,k})} > \varepsilon b_{n} \right\} \\
& \; \; \leqslant \mathbb{P} \left\{\abs{\sum_{k=1}^{n} (X_{n,k}' - \mathbb{E} \, X_{n,k}')} > \frac{\varepsilon b_{n}}{2} \right\} + \mathbb{P} \left\{\abs{\sum_{k=1}^{n} (X_{n,k}'' - \mathbb{E} \, X_{n,k}'')} > \frac{\varepsilon b_{n}}{2} \right\} \\
& \; \;  \leqslant \frac{2^{q}}{\varepsilon^{q} b_{n}^{q}} \mathbb{E} \abs{\sum_{k=1}^{n} (X_{n,k}' - \mathbb{E} \, X_{n,k}')}^{q} + \frac{2}{\varepsilon b_{n}} \mathbb{E} \abs{\sum_{k=1}^{n} (X_{n,k}'' - \mathbb{E} \, X_{n,k}'')} \\
& \; \; \leqslant \frac{2^{q} \alpha_{n}}{\varepsilon^{q} b_{n}^{q}} \sum_{k=1}^{n} \mathbb{E} \, \lvert X_{n,k}' \rvert^{q} + \frac{4}{\varepsilon b_{n}} \sum_{k=1}^{n} \mathbb{E} \, \lvert X_{n,k}'' \rvert \\
& \; \; \leqslant \frac{2^{2q - 1} \alpha_{n}}{\varepsilon^{q} b_{n}^{q}} \sum_{k=1}^{n} \left[\mathbb{E} \, \lvert X_{n,k} \rvert^{q} I_{\left\{\lvert X_{n,k} \rvert \leqslant b_{n} \right\}} + b_{n}^{q} \mathbb{P} \left\{\lvert X_{n,k} \rvert > b_{n} \right\} \right] + \frac{4}{\varepsilon b_{n}} \sum_{k=1}^{n} \mathbb{E} \, \lvert X_{n,k} \rvert I_{\left\{\lvert X_{n,k} \rvert > b_{n} \right\}} \\
& \; \; = \frac{2^{2q - 1} \alpha_{n}}{\varepsilon^{q} b_{n}^{q}} \sum_{k=1}^{n} \int_{0}^{b_{n}^{q}} \mathbb{P} \left\{\lvert X_{n,k} \rvert^{q} > t \right\} \mathrm{d}t + \frac{4}{\varepsilon b_{n}} \sum_{k=1}^{n} \mathbb{E} \, \lvert X_{n,k} \rvert I_{\left\{\lvert X_{n,k} \rvert > b_{n} \right\}}.
\end{split}
\end{equation}
Setting $Y_{n,k}' := g_{t^{1/p}}(X_{n,k})$ and $Y_{n,k}'' = X_{n,k} - Y_{n,k}'$, it follows
\begin{align*}
&\int_{b_{n}^{p}}^{\infty} \mathbb{P} \left\{\abs{\sum_{k=1}^{n} (X_{n,k} - \mathbb{E} \, X_{n,k})} > t^{1/p} \right\} \mathrm{d}t \\
&\qquad \leqslant \int_{b_{n}^{p}}^{\infty} \mathbb{P} \left\{\abs{\sum_{k=1}^{n} (Y_{n,k}' - \mathbb{E} \, Y_{n,k}')} > \frac{t^{1/p}}{2} \right\} \mathrm{d}t + \int_{b_{n}^{p}}^{\infty} \mathbb{P} \left\{\abs{\sum_{k=1}^{n} (Y_{n,k}'' - \mathbb{E} \, Y_{n,k}'')} > \frac{t^{1/p}}{2} \right\} \mathrm{d}t.
\end{align*}
Hence,
\begin{align}
\int_{b_{n}^{p}}^{\infty} & \mathbb{P} \left\{\abs{\sum_{k=1}^{n} \left(Y_{n,k}' - \mathbb{E} \, Y_{n,k}' \right)} > \frac{t^{1/p}}{2} \right\} \mathrm{d}t \notag \\
&\leqslant 2^{q} \int_{b_{n}^{p}}^{\infty} t^{-q/p} \mathbb{E} \abs{\sum_{k=1}^{n} \left(Y_{n,k}' - \mathbb{E} \, Y_{n,k}' \right)}^{q} \mathrm{d}t \notag \\
&\leqslant 2^{q} \alpha_{n} \int_{b_{n}^{p}}^{\infty} t^{-q/p} \sum_{k=1}^{n} \mathbb{E} \lvert Y_{n,k}'\rvert^{q} \mathrm{d}t \notag \\
&\leqslant 2^{2q - 1} \alpha_{n} \int_{b_{n}^{p}}^{\infty} t^{-q/p} \sum_{k=1}^{n} \left[\mathbb{E}\lvert X_{n,k} \rvert^{q} I_{\left\{\abs{X_{n,k}} \leqslant t^{1/p} \right\}} + t^{q/p} \mathbb{P} \left\{\abs{X_{n,k}} > t^{1/p} \right\} \right] \mathrm{d}t \label{eq:2.4} \\
&= 2^{2q - 1} q \alpha_{n} \sum_{k=1}^{n} \int_{b_{n}^{p}}^{\infty} t^{-q/p} \int_{0}^{t^{1/p}} s^{q-1} \, \mathbb{P} \left\{\abs{X_{n,k}} > s \right\} \mathrm{d}s \, \mathrm{d}t \notag \\
&= 2^{2q - 1} q \alpha_{n} \sum_{k=1}^{n} \int_{0}^{\infty} s^{q-1} \, \mathbb{P} \left\{\abs{X_{n,k}} > s \right\} \int_{\max\left(b_{n}^{p}, s^{p} \right)}^{\infty} t^{-q/p} \mathrm{d}t \, \mathrm{d}s \notag \\
&= 2^{2q - 1} q \alpha_{n} \sum_{k=1}^{n} \left[\frac{p b_{n}^{p - q}}{q - p} \int_{0}^{b_{n}} s^{q-1} \, \mathbb{P} \left\{\abs{X_{n,k}} > s \right\} \mathrm{d}s + \frac{p}{q - p} \int_{b_{n}}^{\infty} s^{p-1} \mathbb{P} \left\{\abs{X_{n,k}} > s \right\} \, \mathrm{d}s \right] \notag \\
&= \frac{p \, 2^{2q - 1} \alpha_{n} b_{n}^{p - q}}{q - p} \sum_{k=1}^{n} \int_{0}^{b_{n}^{q}} \mathbb{P} \left\{\abs{X_{n,k}} > t^{1/q} \right\} \mathrm{d}t + \frac{q \, 2^{2q - 1} \alpha_{n}}{q - p} \sum_{k=1}^{n} \int_{b_{n}^{p}}^{\infty} \mathbb{P} \left\{\abs{X_{n,k}} > t^{1/p} \right\} \, \mathrm{d}t. \notag
\end{align}
On the other hand, $\big \lvert Y_{n,k}'' \big \rvert \leqslant \abs{X_{n,k}} I_{\left\{\abs{X_{n,k}} > t^{1/p} \right\}}$, we obtain for every $p > 1$
\begin{equation}\label{eq:2.5}
\begin{split}
\int_{b_{n}^{p}}^{\infty} \mathbb{P} & \left\{\abs{\sum_{k=1}^{n} \left(Y_{n,k}'' - \mathbb{E} \, Y_{n,k}'' \right)} > \frac{t^{1/p}}{2} \right\} \mathrm{d}t \\
&\leqslant 4 \int_{b_{n}^{p}}^{\infty} t^{-1/p} \sum_{k=1}^{n} \mathbb{E} \abs{Y_{n,k}''} \mathrm{d}t \\
&\leqslant 4 \sum_{k=1}^{n} \int_{b_{n}^{p}}^{\infty} t^{-1/p} \mathbb{E} \abs{X_{n,k}} I_{\left\{\abs{X_{n,k}} > t^{1/p} \right\}} \mathrm{d}t \\
&= 4 \sum_{k=1}^{n} \left(\int_{b_{n}^{p}}^{\infty} t^{-1/p} \int_{t^{1/p}}^{\infty} \mathbb{P} \left\{\abs{X_{n,k}} > s \right\} \mathrm{d}s \, \mathrm{d}t + \int_{b_{n}^{p}}^{\infty} \mathbb{P} \left\{\abs{X_{n,k}} > t^{1/p} \right\} \mathrm{d}t \right) \\
&= 4 \sum_{k=1}^{n} \left(\int_{b_{n}}^{\infty} \mathbb{P} \left\{\abs{X_{n,k}} > s \right\} \int_{b_{n}^{p}}^{s^{p}} t^{-1/p} \mathrm{d}t \, \mathrm{d}s + \int_{b_{n}^{p}}^{\infty} \mathbb{P} \left\{\abs{X_{n,k}} > t^{1/p} \right\} \mathrm{d}t \right) \\
&\leqslant 4 \sum_{k=1}^{n} \left(\frac{p}{p - 1}\int_{b_{n}}^{\infty} s^{p-1} \mathbb{P} \left\{\abs{X_{n,k}} > s \right\} \mathrm{d}s + \int_{b_{n}^{p}}^{\infty} \mathbb{P} \left\{\abs{X_{n,k}} > t^{1/p} \right\} \mathrm{d}t \right) \\
&= \frac{4p}{p - 1} \sum_{k=1}^{n} \int_{b_{n}^{p}}^{\infty} \mathbb{P} \left\{\abs{X_{n,k}} > t^{1/p} \right\} \mathrm{d}t.
\end{split}
\end{equation}
Thus, by gathering \eqref{eq:2.2}, \eqref{eq:2.3}, \eqref{eq:2.4} and \eqref{eq:2.5} we get
\begin{equation} \label{eq:2.6}
\begin{split}
& \sum_{n=1}^{\infty} c_{n} \mathbb{E} \left[\frac{\abs{\sum_{k=1}^{n} (X_{n,k} - \mathbb{E} \, X_{n,k})}}{b_{n}} - \varepsilon \right]_{+}^{p} \\
& \quad = \sum_{n=1}^{\infty} \frac{c_{n}}{b_{n}^{p}} \mathbb{E} \left[\abs{\sum_{k=1}^{n} (X_{n,k} - \mathbb{E} \, X_{n,k})} - \varepsilon b_{n} \right]_{+}^{p} \\
& \quad \leqslant \sum_{n=1}^{\infty}  \left(c_{n} \mathbb{P} \left\{\abs{\sum_{k=1}^{n} (X_{n,k} - \mathbb{E} \, X_{n,k})} > \varepsilon b_{n} \right\} + \frac{c_{n}}{b_{n}^{p}} \int_{b_{n}^{p}}^{\infty} \mathbb{P} \left\{\abs{\sum_{k=1}^{n} (X_{n,k} - \mathbb{E} \, X_{n,k})} > t^{1/p} \right\} \mathrm{d}t \right) \\
&\quad \leqslant \left(\frac{2^{2q - 1}}{\varepsilon^{q}} + \frac{p 2^{2q - 1}}{q - p} \right) \sum_{n=1}^{\infty} \sum_{k=1}^{n} \frac{\alpha_{n} c_{n}}{b_{n}^{q}} \int_{0}^{b_{n}^{q}} \mathbb{P} \left\{\abs{X_{n,k}}^{q} > t \right\} \mathrm{d}t + \frac{4}{\varepsilon} \sum_{n=1}^{\infty} \sum_{k=1}^{n} \frac{c_{n}}{b_{n}} \mathbb{E} \, \lvert X_{n,k} \rvert I_{\left\{\lvert X_{n,k} \rvert > b_{n} \right\}} \\
& \qquad + \left(\frac{q \, 2^{2q - 1}}{q - p} + \frac{4p}{p - 1} \right) \sum_{n=1}^{\infty} \sum_{k=1}^{n} \frac{(1 + \alpha_{n}) c_{n}}{b_{n}^{p}} \int_{b_{n}^{p}}^{\infty} \mathbb{P} \left\{\abs{X_{n,k}}^{p} > t \right\} \, \mathrm{d}t \\
& \quad < \infty,
\end{split}
\end{equation}
according to assumptions (a), (b) and (c). The proof is complete.
\end{proof}

\begin{theorem}\label{thr:2.2}
Let $\left\{X_{n,k}, \, 1 \leqslant k \leqslant n, n \geqslant 1 \right\}$ be an array of random variables satisfying $\mathbb{E} \lvert X_{n,k} \rvert < \infty$ for each $1 \leqslant k \leqslant n$ and all $n \geqslant 1$, and verifying \eqref{eq:2.1} for a $q > 1$ and some sequence $\{\alpha_{n} \}$ of positive numbers. If $\{b_{n} \}$, $\{c_{n} \}$ are real sequences of positive numbers such that \\

\noindent \textnormal{(a)} $\sum_{n=1}^{\infty} \sum_{k=1}^{n} \alpha_{n} c_{n} b_{n}^{-q} \int_{0}^{b_{n}^{q}} \mathbb{P} \left\{\lvert X_{n,k} \rvert^{q} > t \right\} \, \mathrm{d}t < \infty$, \\

\noindent \textnormal{(b)} $\sum_{n=1}^{\infty} \sum_{k=1}^{n} c_{n} \mathbb{E} \, \lvert X_{n,k} \rvert I_{\left\{\lvert X_{n,k} \rvert > b_{n} \right\}}/b_{n} < \infty$, \\

\noindent \textnormal{(c)} $\sum_{n=1}^{\infty} \sum_{k=1}^{n} (\alpha_{n} c_{n}/b_{n}) \int_{b_{n}}^{\infty} \mathbb{P} \left\{\lvert X_{n,k} \rvert > t \right\} \, \mathrm{d}t < \infty$, \\

\noindent then
\begin{equation*}
\sum_{n=1}^{\infty} c_{n} \mathbb{E} \left[\frac{\abs{\sum_{k=1}^{n} (X_{n,k} - \mathbb{E} \, X_{n,k})}}{b_{n}} - \varepsilon \right]_{+} < \infty
\end{equation*}
for all $\varepsilon > 0$.
\end{theorem}

\begin{proof}
All steps in the proof of Theorem~\ref{thr:2.1} remains true for $p = 1$ except the upper bound \eqref{eq:2.5}. Supposing $Y_{n,k}' := g_{t}(X_{n,k})$ and $Y_{n,k}'' = X_{n,k} - Y_{n,k}'$ we have, for any $t \geqslant b_{n}$,
\begin{align*}
\abs{\sum_{k=1}^{n} \left(Y_{n,k}'' - \mathbb{E} \, Y_{n,k}'' \right)} &\leqslant \sum_{k=1}^{n} \left(\lvert Y_{n,k}'' \rvert + \mathbb{E} \lvert Y_{n,k}'' \rvert \right) \\
&\leqslant \sum_{k=1}^{n} \left(\abs{X_{n,k}} I_{\left\{\abs{X_{n,k}} > t \right\}} + \mathbb{E} \abs{X_{n,k}} I_{\left\{\abs{X_{n,k}} > t \right\}} \right) \\
&\leqslant \sum_{k=1}^{n} \left(\abs{X_{n,k}} I_{\left\{\abs{X_{n,k}} > b_{n} \right\}} + \mathbb{E} \abs{X_{n,k}} I_{\left\{\abs{X_{n,k}} > b_{n} \right\}} \right).
\end{align*}
Hence,
\begin{equation}\label{eq:2.7}
\begin{split}
\int_{b_{n}}^{\infty} \mathbb{P} & \left\{\abs{\sum_{k=1}^{n} \left(Y_{n,k}'' - \mathbb{E} \, Y_{n,k}'' \right)} > \frac{t}{2} \right\} \mathrm{d}t \\
&\leqslant \int_{b_{n}}^{\infty} \mathbb{P} \left\{\sum_{k=1}^{n} \left(\abs{X_{n,k}} I_{\left\{\abs{X_{n,k}} > b_{n} \right\}} + \mathbb{E} \abs{X_{n,k}} I_{\left\{\abs{X_{n,k}} > b_{n} \right\}} \right) > \frac{t}{2} \right\} \mathrm{d}t \\
&= 2 \int_{\frac{b_{n}}{2}}^{\infty} \mathbb{P} \left\{\sum_{k=1}^{n} \left(\abs{X_{n,k}} I_{\left\{\abs{X_{n,k}} > b_{n} \right\}} + \mathbb{E} \abs{X_{n,k}} I_{\left\{\abs{X_{n,k}} > b_{n} \right\}} \right) > s \right\} \mathrm{d}s \\
&\leqslant 2 \, \mathbb{E} \left[\sum_{k=1}^{n} \left(\abs{X_{n,k}} I_{\left\{\abs{X_{n,k}} > b_{n} \right\}} + \mathbb{E} \abs{X_{n,k}} I_{\left\{\abs{X_{n,k}} > b_{n} \right\}} \right) \right] \\
&=4 \sum_{k=1}^{n} \mathbb{E} \abs{X_{n,k}} I_{\left\{\abs{X_{n,k}} > b_{n} \right\}}
\end{split}
\end{equation}
and
\begin{align*}
& c_{n} \mathbb{E} \left[\frac{\abs{\sum_{k=1}^{n} (X_{n,k} - \mathbb{E} \, X_{n,k})}}{b_{n}} - \varepsilon \right]_{+} \\
&\quad = \frac{c_{n}}{b_{n}} \mathbb{E} \left[\abs{\sum_{k=1}^{n} (X_{n,k} - \mathbb{E} \, X_{n,k})} - \varepsilon b_{n} \right]_{+} \\
&\quad \leqslant \left(\frac{2^{2q - 1}}{\varepsilon^{q}} + \frac{2^{2q - 1}}{q - 1} \right) \sum_{k=1}^{n} \frac{\alpha_{n} c_{n}}{b_{n}^{q}} \int_{0}^{b_{n}^{q}} \mathbb{P} \left\{\abs{X_{n,k}}^{q} > t \right\} \mathrm{d}t + \left(\frac{4}{\varepsilon} + 4 \right) \sum_{k=1}^{n} \frac{c_{n}}{b_{n}} \mathbb{E} \, \lvert X_{n,k} \rvert I_{\left\{\lvert X_{n,k} \rvert > b_{n} \right\}} \\
& \qquad + \frac{q \, 2^{2q - 1}}{q - 1} \sum_{k=1}^{n} \frac{\alpha_{n} c_{n}}{b_{n}} \int_{b_{n}}^{\infty} \mathbb{P} \left\{\abs{X_{n,k}} > t \right\} \, \mathrm{d}t
\end{align*}
by employing \eqref{eq:2.2}, \eqref{eq:2.4} with $p = 1$ and \eqref{eq:2.3}, \eqref{eq:2.7}. The thesis is established.
\end{proof}

The next two results, give us conditions for the convergence of \eqref{eq:1.3} under the assumption that, for every $t > 0$, the array of random variables $\left\{g_{t}(X_{n,k}), \, 1 \leqslant k \leqslant n, n \geqslant 1 \right\}$ satisfies a Rosenthal type inequality. Specifically, we shall admit that there are sequences of positive numbers $\{\beta_{n} \}$ and $\{\xi_{n} \}$ such that for some $q > 2$,
\begin{equation}\label{eq:2.8}
\mathbb{E} \abs{\sum_{k=1}^{n} \left[g_{t}(X_{n,k}) - \mathbb{E} \, g_{t}(X_{n,k}) \right]}^{q} \leqslant \beta_{n} \sum_{k=1}^{n} \mathbb{E} \, \lvert g_{t}(X_{n,k}) \rvert^{q} + \xi_{n} \left[\sum_{k=1}^{n} \mathbb{E} \, \lvert g_{t}(X_{n,k}) \rvert^{2} \right]^{q/2}
\end{equation}
for all $n \geqslant 1$ and $t > 0$.

\begin{theorem}\label{thr:2.3}
Let $p > 1$, $\left\{X_{n,k}, \, 1 \leqslant k \leqslant n, n \geqslant 1 \right\}$ be an array of random variables satisfying $\mathbb{E} \lvert X_{n,k} \rvert^{p} < \infty$ for each $1 \leqslant k \leqslant n$ and all $n \geqslant 1$, and verifying \eqref{eq:2.8} for a $q > \max\{p,2\}$ and some sequences $\{\beta_{n} \}$, $\{\xi_{n} \}$ of positive numbers. If $\{b_{n} \}$, $\{c_{n} \}$ are real sequences of positive numbers such that \\

\noindent \textnormal{(a)} $\sum_{n=1}^{\infty} \sum_{k=1}^{n} \beta_{n} c_{n} b_{n}^{-q} \int_{0}^{b_{n}^{q}} \mathbb{P} \left\{\lvert X_{n,k} \rvert^{q} > t \right\} \, \mathrm{d}t < \infty$, \\

\noindent \textnormal{(b)} $\sum_{n=1}^{\infty} \xi_{n} c_{n} b_{n}^{-p} \int_{0}^{b_{n}^{p - q}} \left(\sum_{k=1}^{n} \int_{0}^{t^{2/(p - q)}} \mathbb{P} \{X_{n,k}^{2} > s \} \mathrm{d}s \right)^{q/2} \mathrm{d}t < \infty$, \\

\noindent \textnormal{(c)} $\sum_{n=1}^{\infty} \xi_{n} c_{n} b_{n}^{-q} \left(\sum_{k=1}^{n} \int_{0}^{b_{n}^{2}} \mathbb{P} \{X_{n,k}^{2} > t \} \, \mathrm{d}t \right)^{q/2} < \infty$, \\

\noindent \textnormal{(d)} $\sum_{n=1}^{\infty} \sum_{k=1}^{n} c_{n} \mathbb{E} \, \lvert X_{n,k} \rvert I_{\left\{\lvert X_{n,k} \rvert > b_{n} \right\}}/b_{n} < \infty$, \\

\noindent \textnormal{(e)} $\sum_{n=1}^{\infty} \sum_{k=1}^{n} (1 + \beta_{n}) c_{n} b_{n}^{-p} \int_{b_{n}^{p}}^{\infty} \mathbb{P} \left\{\lvert X_{n,k} \rvert^{p} > t \right\} \, \mathrm{d}t < \infty$, \\

\noindent then
\begin{equation*}
\sum_{n=1}^{\infty} c_{n} \mathbb{E} \left[\frac{\abs{\sum_{k=1}^{n} (X_{n,k} - \mathbb{E} \, X_{n,k})}}{b_{n}} - \varepsilon \right]_{+}^{p} < \infty
\end{equation*}
for all $\varepsilon > 0$.
\end{theorem}

\begin{proof}
The proof follows in exactly the same manner as the proof of Theorem~\ref{thr:2.1} except for upper bounds \eqref{eq:2.3} and \eqref{eq:2.4} which must be replaced. Letting $X_{n,k}' := g_{b_{n}}(X_{n,k})$ and $X_{n,k}'' = X_{n,k} - X_{n,k}'$, assumption \eqref{eq:2.8} ensures
\begin{align}
& \mathbb{P} \left\{\abs{\sum_{k=1}^{n} (X_{n,k} - \mathbb{E} \, X_{n,k})} > \varepsilon b_{n} \right\} \notag \\
& \; \; \leqslant \mathbb{P} \left\{\abs{\sum_{k=1}^{n} (X_{n,k}' - \mathbb{E} \, X_{n,k}')} > \frac{\varepsilon b_{n}}{2} \right\} + \mathbb{P} \left\{\abs{\sum_{k=1}^{n} (X_{n,k}'' - \mathbb{E} \, X_{n,k}'')} > \frac{\varepsilon b_{n}}{2} \right\} \notag \\
& \; \; \leqslant \frac{2^{q}}{\varepsilon^{q} b_{n}^{q}} \mathbb{E} \abs{\sum_{k=1}^{n} (X_{n,k}' - \mathbb{E} \, X_{n,k}')}^{q} + \frac{2}{\varepsilon b_{n}} \mathbb{E} \abs{\sum_{k=1}^{n} (X_{n,k}'' - \mathbb{E} \, X_{n,k}'')} \notag \\
& \; \; \leqslant \frac{2^{q} \beta_{n}}{\varepsilon^{q} b_{n}^{q}} \sum_{k=1}^{n} \mathbb{E} \, \lvert X_{n,k}' \rvert^{q} + \frac{2^{q} \xi_{n}}{\varepsilon^{q} b_{n}^{q}} \left(\sum_{k=1}^{n} \mathbb{E} \, \lvert X_{n,k}' \rvert^{2} \right)^{q/2} + \frac{4}{\varepsilon b_{n}} \sum_{k=1}^{n} \mathbb{E} \, \lvert X_{n,k}'' \rvert \label{eq:2.9} \\
& \; \; \leqslant \frac{2^{2q - 1} \beta_{n}}{\varepsilon^{q} b_{n}^{q}} \sum_{k=1}^{n} \left[\mathbb{E} \, \lvert X_{n,k} \rvert^{q} I_{\left\{\lvert X_{n,k} \rvert \leqslant b_{n} \right\}} + b_{n}^{q} \mathbb{P} \left\{\lvert X_{n,k} \rvert > b_{n} \right\} \right] + \frac{4}{\varepsilon b_{n}} \sum_{k=1}^{n} \mathbb{E} \, \lvert X_{n,k} \rvert I_{\left\{\lvert X_{n,k} \rvert > b_{n} \right\}} \notag \\
& \; \; \qquad + \frac{2^{3q/2} \xi_{n}}{\varepsilon^{q} b_{n}^{q}} \left(\sum_{k=1}^{n} \mathbb{E} \, X_{n,k}^{2} I_{\left\{\lvert X_{n,k} \rvert \leqslant b_{n} \right\}} + b_{n}^{2} \sum_{k=1}^{n} \mathbb{P} \left\{\lvert X_{n,k} \rvert > b_{n} \right\} \right)^{q/2} \notag \\
& \; \; = \frac{2^{2q - 1} \beta_{n}}{\varepsilon^{q} b_{n}^{q}} \sum_{k=1}^{n} \int_{0}^{b_{n}^{q}} \mathbb{P} \left\{\lvert X_{n,k} \rvert^{q} > t \right\} \mathrm{d}t + \frac{4}{\varepsilon b_{n}} \sum_{k=1}^{n} \mathbb{E} \, \lvert X_{n,k} \rvert I_{\left\{\lvert X_{n,k} \rvert > b_{n} \right\}} \notag \\
& \; \; \qquad + \frac{2^{3q/2} \xi_{n}}{\varepsilon^{q} b_{n}^{q}} \left(\sum_{k=1}^{n} \int_{0}^{b_{n}^{2}} \mathbb{P} \left\{X_{n,k}^{2} > t \right\} \, \mathrm{d}t \right)^{q/2} \notag.
\end{align}
On the other hand, considering $Y_{n,k}' := g_{t^{1/p}}(X_{n,k})$ we have
\begin{equation}\label{eq:2.10}
\begin{split}
& \int_{b_{n}^{p}}^{\infty} \mathbb{P} \left\{\abs{\sum_{k=1}^{n} \left(Y_{n,k}' - \mathbb{E} \, Y_{n,k}' \right)} > \frac{t^{1/p}}{2} \right\} \mathrm{d}t \\
& \; \; \leqslant 2^{q} \int_{b_{n}^{p}}^{\infty} t^{-q/p} \mathbb{E} \abs{\sum_{k=1}^{n} \left(Y_{n,k}' - \mathbb{E} \, Y_{n,k}' \right)}^{q} \mathrm{d}t \\
& \; \; \leqslant 2^{q} \beta_{n} \int_{b_{n}^{p}}^{\infty} t^{-q/p} \sum_{k=1}^{n} \mathbb{E} \lvert Y_{n,k}'\rvert^{q} \mathrm{d}t + 2^{q} \xi_{n} \int_{b_{n}^{p}}^{\infty} t^{-q/p} \left(\sum_{k=1}^{n} \mathbb{E} \lvert Y_{n,k}'\rvert^{2} \right)^{q/2} \mathrm{d}t \\
& \; \; \leqslant \frac{p \, 2^{2q - 1} \beta_{n} b_{n}^{p - q}}{q - p} \sum_{k=1}^{n} \int_{0}^{b_{n}^{q}} \mathbb{P} \left\{\abs{X_{n,k}} > t^{1/q} \right\} \mathrm{d}t + \frac{q \, 2^{2q - 1} \beta_{n}}{q - p} \sum_{k=1}^{n} \int_{b_{n}^{p}}^{\infty} \mathbb{P} \left\{\abs{X_{n,k}} > t^{1/p} \right\} \, \mathrm{d}t \\
& \; \; \qquad + 2^{3q/2} \xi_{n} \int_{b_{n}^{p}}^{\infty} t^{-q/p} \left[\sum_{k=1}^{n} \left(\mathbb{E} \, X_{n,k}^{2} I_{\left\{\lvert X_{n,k} \rvert \leqslant t^{1/p} \right\}} + t^{2/p} \mathbb{P} \left\{\abs{X_{n,k}} > t^{1/p} \right\} \right) \right]^{q/2} \mathrm{d}t  \\
& \; \; = \frac{p \, 2^{2q - 1} \beta_{n} b_{n}^{p - q}}{q - p} \sum_{k=1}^{n} \int_{0}^{b_{n}^{q}} \mathbb{P} \left\{\abs{X_{n,k}} > t^{1/q} \right\} \mathrm{d}t + \frac{q \, 2^{2q - 1} \beta_{n}}{q - p} \sum_{k=1}^{n} \int_{b_{n}^{p}}^{\infty} \mathbb{P} \left\{\abs{X_{n,k}} > t^{1/p} \right\} \, \mathrm{d}t \\
& \; \; \qquad + 2^{3q/2} \xi_{n} \int_{b_{n}^{p}}^{\infty} t^{-q/p} \left(\int_{0}^{t^{2/p}} \sum_{k=1}^{n} \mathbb{P} \left\{X_{n,k}^{2} > s \right\} \mathrm{d}s \right)^{q/2} \mathrm{d}t \\
& \; \; = \frac{p \, 2^{2q - 1} \beta_{n} b_{n}^{p - q}}{q - p} \sum_{k=1}^{n} \int_{0}^{b_{n}^{q}} \mathbb{P} \left\{\abs{X_{n,k}} > t^{1/q} \right\} \mathrm{d}t + \frac{q \, 2^{2q - 1} \beta_{n}}{q - p} \sum_{k=1}^{n} \int_{b_{n}^{p}}^{\infty} \mathbb{P} \left\{\abs{X_{n,k}} > t^{1/p} \right\} \, \mathrm{d}t \\
& \; \; \qquad + \frac{p 2^{3q/2} \xi_{n}}{q - p} \int_{0}^{b_{n}^{p - q}} \left(\sum_{k=1}^{n} \int_{0}^{v^{2/(p - q)}} \mathbb{P} \left\{X_{n,k}^{2} > s \right\} \mathrm{d}s \right)^{q/2} \mathrm{d}v.
\end{split}
\end{equation}
Employing \eqref{eq:2.2}, \eqref{eq:2.5}, \eqref{eq:2.9} and \eqref{eq:2.10} as in \eqref{eq:2.6} the conclusion follows. The proof is complete.
\end{proof}

\begin{theorem}\label{thr:2.4}
Let $\left\{X_{n,k}, \, 1 \leqslant k \leqslant n, n \geqslant 1 \right\}$ be an array of random variables satisfying $\mathbb{E} \lvert X_{n,k} \rvert < \infty$ for each $1 \leqslant k \leqslant n$ and all $n \geqslant 1$, and verifying \eqref{eq:2.8} for a $q > 2$ and some sequences $\{\beta_{n} \}$, $\{\xi_{n} \}$ of positive numbers. If $\{b_{n} \}$, $\{c_{n} \}$ are real sequences of positive numbers such that \\

\noindent \textnormal{(a)} $\sum_{n=1}^{\infty} \sum_{k=1}^{n} \beta_{n} c_{n} b_{n}^{-q} \int_{0}^{b_{n}^{q}} \mathbb{P} \left\{\lvert X_{n,k} \rvert^{q} > t \right\} \, \mathrm{d}t < \infty$, \\

\noindent \textnormal{(b)} $\sum_{n=1}^{\infty} (\xi_{n} c_{n}/b_{n}) \int_{0}^{b_{n}^{1 - q}} \left(\sum_{k=1}^{n} \int_{0}^{t^{2/(1 - q)}} \mathbb{P} \{X_{n,k}^{2} > s \} \mathrm{d}s \right)^{q/2} \mathrm{d}t < \infty$ \\

\noindent \textnormal{(c)} $\sum_{n=1}^{\infty} \xi_{n} c_{n} b_{n}^{-q} \left(\sum_{k=1}^{n} \int_{0}^{b_{n}^{2}} \mathbb{P} \{X_{n,k}^{2} > t \} \, \mathrm{d}t \right)^{q/2} < \infty$, \\

\noindent \textnormal{(d)} $\sum_{n=1}^{\infty} \sum_{k=1}^{n} c_{n} \mathbb{E} \, \lvert X_{n,k} \rvert I_{\left\{\lvert X_{n,k} \rvert > b_{n} \right\}}/b_{n} < \infty$, \\

\noindent \textnormal{(e)} $\sum_{n=1}^{\infty} \sum_{k=1}^{n} (\beta_{n} c_{n}/b_{n}) \int_{b_{n}}^{\infty} \mathbb{P} \left\{\lvert X_{n,k} \rvert > t \right\} \, \mathrm{d}t < \infty$, \\

\noindent then
\begin{equation*}
\sum_{n=1}^{\infty} c_{n} \mathbb{E} \left[\frac{\abs{\sum_{k=1}^{n} (X_{n,k} - \mathbb{E} \, X_{n,k})}}{b_{n}} - \varepsilon \right]_{+} < \infty
\end{equation*}
for all $\varepsilon > 0$.
\end{theorem}

\begin{proof}
The thesis is a consequence of \eqref{eq:2.2}, \eqref{eq:2.10} with $p=1$ and \eqref{eq:2.7}, \eqref{eq:2.9}.
\end{proof}

\begin{remark}
Notice that if $\left\{X_{n,k}, \, 1 \leqslant k \leqslant n, n \geqslant 1 \right\}$ is an array of row-wise extended negatively dependent random variables with dominating sequence $\{M_{n}, n \geqslant 1 \}$ (see \cite{Lita16}), then \eqref{eq:2.8} holds with $q \geqslant 2$ and $\beta_{n} = \xi_{n} = C(q) (1 + M_{n})$ with $C(q)$ a positive constant depending only on $q$ (see Lemma 2 of \cite{Lita16}); further, \eqref{eq:2.1} still holds for these dependent structures with $1 \leqslant q \leqslant 2$ and $\alpha_{n} = C(q)(1 + M_{n})$, where $C(q) > 0$ depends only on $q$.
\end{remark}

Supposing $0 < p \leqslant 1$, $\varepsilon > 0$ and $b_{n}$ a real sequence of positive numbers, Lemma 2.1 of \cite{Sung09} and elementary inequality $(x + y)^{p} \leqslant x^{p} + y^{p}$, $x,y \geqslant 0$ lead us to
\begin{equation*}
\mathbb{E} \left(\abs{\sum_{k=1}^{n} X_{n,k} } - \varepsilon b_{n} \right)_{+}^{p} \leqslant \mathbb{E} \left(\abs{\sum_{k=1}^{n} X_{n,k} I_{\left\{\lvert X_{n,k} \rvert \leqslant b_{n} \right\}}} - \varepsilon b_{n} \right)_{+}^{p} + \mathbb{E} \abs{\sum_{k=1}^{n} X_{n,k} I_{\left\{\lvert X_{n,k} \rvert > b_{n} \right\}}}^{p}.
\end{equation*}
By taking $q > p$, we obtain
\begin{align*}
& \mathbb{E} \left(\abs{\sum_{k=1}^{n} X_{n,k} I_{\left\{\lvert X_{n,k} \rvert \leqslant b_{n} \right\}}} - \varepsilon b_{n} \right)_{+}^{p} \\
&\qquad \leqslant b_{n}^{p} \mathbb{P} \left\{\abs{\sum_{k=1}^{n} X_{n,k} I_{\left\{\lvert X_{n,k} \rvert \leqslant b_{n} \right\}}} > \varepsilon b_{n} \right\} + \int_{b_{n}^{p}}^{\infty} \mathbb{P} \left\{\abs{\sum_{k=1}^{n} X_{n,k} I_{\left\{\lvert X_{n,k} \rvert \leqslant b_{n} \right\}}} > t^{1/p} \right\} \mathrm{d}t \\
&\qquad \leqslant \varepsilon^{-q} b_{n}^{p-q} \mathbb{E} \abs{\sum_{k=1}^{n} X_{n,k} I_{\left\{\lvert X_{n,k} \rvert \leqslant b_{n} \right\}}}^{q} + \mathbb{E} \abs{\sum_{k=1}^{n} X_{n,k} I_{\left\{\lvert X_{n,k} \rvert \leqslant b_{n} \right\}}}^{q} \int_{b_{n}^{p}}^{\infty} t^{-q/p} \mathrm{d}t \\
&\qquad = \varepsilon^{-q} b_{n}^{p-q} \mathbb{E} \abs{\sum_{k=1}^{n} X_{n,k} I_{\left\{\lvert X_{n,k} \rvert \leqslant b_{n} \right\}}}^{q} + \frac{p b_{n}^{p-q}}{q - p} \mathbb{E} \abs{\sum_{k=1}^{n} X_{n,k} I_{\left\{\lvert X_{n,k} \rvert \leqslant b_{n} \right\}}}^{q} \\
&\qquad = \left(\varepsilon^{-q} + \frac{p}{q - p} \right) b_{n}^{p-q} \mathbb{E} \abs{\sum_{k=1}^{n} X_{n,k} I_{\left\{\lvert X_{n,k} \rvert \leqslant b_{n} \right\}}}^{q}
\end{align*}
which yields
\begin{equation}\label{eq:2.11}
\mathbb{E} \left(\abs{\sum_{k=1}^{n} X_{n,k} } - \varepsilon b_{n} \right)_{+}^{p} \leqslant \left(\varepsilon^{-q} + \frac{p}{q - p} \right) b_{n}^{p-q} \mathbb{E} \abs{\sum_{k=1}^{n} X_{n,k} I_{\left\{\lvert X_{n,k} \rvert \leqslant b_{n} \right\}}}^{q} + \mathbb{E} \abs{\sum_{k=1}^{n} X_{n,k} I_{\left\{\lvert X_{n,k} \rvert > b_{n} \right\}}}^{p}.
\end{equation}
Hence, we can still announce the result hereinafter whose proof follows from inequality \eqref{eq:2.11}; we omit the details.

\begin{theorem}\label{thr:2.5}
Let $0 < p < 1$, $\left\{X_{n,k}, \, 1 \leqslant k \leqslant n, n \geqslant 1 \right\}$ be an array of random variables satisfying $\mathbb{E} \lvert X_{n,k} \rvert^{p} < \infty$ for each $1 \leqslant k \leqslant n$ and all $n \geqslant 1$. If $\{b_{n} \}$, $\{c_{n} \}$ are real sequences of positive numbers such that \\

\noindent \textnormal{(a)} $\sum_{n=1}^{\infty} \sum_{k=1}^{n} c_{n} b_{n}^{-q} \mathbb{E} \lvert X_{n,k} \rvert^{q} I_{\left\{\lvert X_{n,k} \rvert \leqslant b_{n} \right\}} < \infty$ for some $p < q \leqslant 1$, \\

\noindent \textnormal{(b)} $\sum_{n=1}^{\infty} \sum_{k=1}^{n} c_{n} b_{n}^{-p} \mathbb{E} \, \lvert X_{n,k} \rvert^{p} I_{\left\{\lvert X_{n,k} \rvert > b_{n} \right\}} < \infty$, \\

\noindent then
\begin{equation*}
\sum_{n=1}^{\infty} c_{n} \mathbb{E} \left[\frac{\abs{\sum_{k=1}^{n} X_{n,k}}}{b_{n}} - \varepsilon \right]_{+}^{p} < \infty
\end{equation*}
for all $\varepsilon > 0$.
\end{theorem}

\begin{remark}
Under the assumptions of Theorem~\ref{thr:2.1} (or Theorem~\ref{thr:2.3}) we obviously have, for any $0 < r \leqslant p$,
\begin{equation}\label{eq:2.12}
\sum_{n=1}^{\infty} c_{n} \mathbb{E} \left[\frac{\abs{\sum_{k=1}^{n} (X_{n,k} - \mathbb{E} \, X_{n,k})}}{b_{n}} - \varepsilon \right]_{+}^{r} < \infty
\end{equation}
for all $\varepsilon > 0$, because
\begin{align*}
&\mathbb{E} \left(\abs{\sum_{k=1}^{n} (X_{n,k} - \mathbb{E} \, X_{n,k})} - \varepsilon b_{n} \right)_{+}^{r} \\
&\quad \leqslant b_{n}^{r} \mathbb{P} \left\{\abs{\sum_{k=1}^{n} (X_{n,k} - \mathbb{E} \, X_{n,k})} > \varepsilon b_{n} \right\} + \int_{b_{n}^{r}}^{\infty} \mathbb{P} \left\{\abs{\sum_{k=1}^{n} (X_{n,k} - \mathbb{E} \, X_{n,k})} > t^{1/r} \right\} \mathrm{d}t \\
& \quad = b_{n}^{r} \mathbb{P} \left\{\abs{\sum_{k=1}^{n} (X_{n,k} - \mathbb{E} \, X_{n,k})} > \varepsilon b_{n} \right\} + \frac{r}{p} \int_{b_{n}^{p}}^{\infty} s^{r/p - 1} \mathbb{P} \left\{\abs{\sum_{k=1}^{n} (X_{n,k} - \mathbb{E} \, X_{n,k})} > s^{1/p} \right\} \mathrm{d}s \quad \left(t = s^{r/p} \right) \\
& \quad \leqslant b_{n}^{r} \mathbb{P} \left\{\abs{\sum_{k=1}^{n} (X_{n,k} - \mathbb{E} \, X_{n,k})} > \varepsilon b_{n} \right\} + b_{n}^{r - p} \int_{b_{n}^{p}}^{\infty} \mathbb{P} \left\{\abs{\sum_{k=1}^{n} (X_{n,k} - \mathbb{E} \, X_{n,k})} > s^{1/p} \right\} \mathrm{d}s
\end{align*}
and
\begin{align*}
&\sum_{n=1}^{\infty} c_{n} \mathbb{E} \left[\frac{\abs{\sum_{k=1}^{n} (X_{n,k} - \mathbb{E} \, X_{n,k})}}{b_{n}} - \varepsilon \right]_{+}^{r} \\
& \quad \leqslant \sum_{n=1}^{\infty} \left(c_{n} \mathbb{P} \left\{\abs{\sum_{k=1}^{n} (X_{n,k} - \mathbb{E} \, X_{n,k})} > \varepsilon b_{n} \right\} + \frac{c_{n}}{b_{n}^{p}} \int_{b_{n}^{p}}^{\infty} \mathbb{P} \left\{\abs{\sum_{k=1}^{n} (X_{n,k} - \mathbb{E} \, X_{n,k})} > s^{1/p} \right\} \mathrm{d}s \right).
\end{align*}
In the same way, \eqref{eq:2.12} holds for every $0 < r \leqslant 1$ whenever the assumptions of Theorem~\ref{thr:2.2} (or Theorem~\ref{thr:2.4}) are met.
\end{remark}

\section{Applications}

It is straightforward to see that
\begin{gather}
\int_{u^{p}}^{\infty} \mathbb{P} \left\{\lvert X_{n,k} \rvert^{p} > t \right\} \mathrm{d}t \leqslant \mathbb{E} \, \lvert X_{n,k} \rvert^{p} I_{\left\{\lvert X_{n,k} \rvert > u \right\}}, \label{eq:3.1} \\
\mathbb{P} \left\{\lvert X_{n,k} \rvert > u \right\} \leqslant \frac{\mathbb{E} \, \lvert X_{n,k} \rvert^{p} I_{\left\{\lvert X_{n,k} \rvert > u \right\}}}{u^{p}} \label{eq:3.2}
\end{gather}
and
\begin{equation}\label{eq:3.3}
\int_{0}^{u^{q}} \mathbb{P} \left\{\lvert X_{n,k} \rvert^{q} > t \right\} \mathrm{d}t \leqslant u^{q- p} \mathbb{E} \, \lvert X_{n,k} \rvert^{p} I_{\left\{\lvert X_{n,k} \rvert > u \right\}} + \mathbb{E} \, \lvert X_{n,k} \rvert^{q} I_{\left\{\lvert X_{n,k} \rvert \leqslant u \right\}}
\end{equation}
for any $p,q,u > 0$. Thus, if $\{\alpha_{n} \}$ is a constant sequence then both Theorems~\ref{thr:2.1} and~\ref{thr:2.2} can be gathered in the following result.

\begin{corollary}\label{cor:3.1}
Let $p \geqslant 1$, $\left\{X_{n,k}, \, 1 \leqslant k \leqslant n, n \geqslant 1 \right\}$ be an array of random variables satisfying $\mathbb{E} \lvert X_{n,k} \rvert^{p} < \infty$ for each $1 \leqslant k \leqslant n$ and all $n \geqslant 1$, and verifying \eqref{eq:2.1} for a $q > p$ and some constant sequence $\{\alpha_{n} \}$. If $\{b_{n} \}$, $\{c_{n} \}$ are real sequences of positive numbers such that \\

\noindent \textnormal{(i)} $\sum_{n=1}^{\infty} \sum_{k=1}^{n} c_{n} b_{n}^{-q} \mathbb{E} \, \lvert X_{n,k} \rvert^{q} I_{\left\{\lvert X_{n,k} \rvert \leqslant b_{n} \right\}} < \infty$, \\

\noindent \textnormal{(ii)} $\sum_{n=1}^{\infty} \sum_{k=1}^{n} c_{n} b_{n}^{-p} \mathbb{E} \, \lvert X_{n,k} \rvert^{p} I_{\left\{\lvert X_{n,k} \rvert > b_{n} \right\}} < \infty$, \\

\noindent then
\begin{equation*}
\sum_{n=1}^{\infty} c_{n} \mathbb{E} \left[\frac{\abs{\sum_{k=1}^{n} (X_{n,k} - \mathbb{E} \, X_{n,k})}}{b_{n}} - \varepsilon \right]_{+}^{p} < \infty
\end{equation*}
for all $\varepsilon > 0$.
\end{corollary}

\begin{proof}
Since $\{\alpha_{n} \}$ is a constant sequence, (ii) ensures assumption (c) of Theorems~\ref{thr:2.1} and~\ref{thr:2.2} via \eqref{eq:3.1}. According to \eqref{eq:3.3}, (i) and (ii) together guarantee assumption (a) of Theorems~\ref{thr:2.1} and~\ref{thr:2.2}. Finally, assumption (b) of Theorems~\ref{thr:2.1} and~\ref{thr:2.2} follows from (ii) by noting that
\begin{equation*}
\frac{\lvert X_{n,k} \rvert I_{\left\{\lvert X_{n,k} \rvert > b_{n} \right\}}}{b_{n}} \leqslant \frac{\lvert X_{n,k} \rvert^{p} I_{\left\{\lvert X_{n,k} \rvert > b_{n} \right\}}}{b_{n}^{p}}.
\end{equation*}
Hence, Corollary~\ref{cor:3.1} is proved.
\end{proof}

Similarly, we can also join Theorems~\ref{thr:2.3} and~\ref{thr:2.4} when sequences $\{\beta_{n} \}$ and $\{\xi_{n} \}$ are constant.

\begin{corollary}\label{cor:3.2}
Let $p \geqslant 1$, $\left\{X_{n,k}, \, 1 \leqslant k \leqslant n, n \geqslant 1 \right\}$ be an array of random variables satisfying $\mathbb{E} \lvert X_{n,k} \rvert^{p} < \infty$ for each $1 \leqslant k \leqslant n$ and all $n \geqslant 1$, and verifying \eqref{eq:2.8} for a $q > \max\{p,2\}$ and some constant sequences $\{\beta_{n} \}$, $\{\xi_{n} \}$. If $\{b_{n} \}$, $\{c_{n} \}$ are real sequences of positive numbers such that \\

\noindent \textnormal{(i)} $\sum_{n=1}^{\infty} \sum_{k=1}^{n} c_{n} b_{n}^{-q} \mathbb{E} \, \lvert X_{n,k} \rvert^{q} I_{\left\{\lvert X_{n,k} \rvert \leqslant b_{n} \right\}} < \infty$, \\

\noindent \textnormal{(ii)} $\sum_{n=1}^{\infty} \sum_{k=1}^{n} c_{n} b_{n}^{-p} \mathbb{E} \, \lvert X_{n,k} \rvert^{p} I_{\left\{\lvert X_{n,k} \rvert > b_{n} \right\}} < \infty$, \\

\noindent \textnormal{(iii)} $\sum_{n=1}^{\infty} c_{n} b_{n}^{-pq/2} \left(\sum_{k=1}^{n} \mathbb{E} \, \lvert X_{n,k} \rvert^{p} I_{\left\{\lvert X_{n,k} \rvert > b_{n} \right\}} \right)^{q/2}  < \infty$, \\

\noindent \textnormal{(iv)} $\sum_{n=1}^{\infty} c_{n} b_{n}^{-q} \left(\sum_{k=1}^{n} \mathbb{E} \, X_{n,k}^{2} I_{\left\{\lvert X_{n,k} \rvert \leqslant b_{n} \right\}} \right)^{q/2}  < \infty$, \\

\noindent then
\begin{equation*}
\sum_{n=1}^{\infty} c_{n} \mathbb{E} \left[\frac{\abs{\sum_{k=1}^{n} (X_{n,k} - \mathbb{E} \, X_{n,k})}}{b_{n}} - \varepsilon \right]_{+}^{p} < \infty
\end{equation*}
for all $\varepsilon > 0$.
\end{corollary}

\begin{proof}
From (iii) and (iv), we have that conditions (b), (c) of Theorems~\ref{thr:2.3} and~\ref{thr:2.4} are verified since
\begin{align*}
& b_{n}^{-p} \int_{0}^{b_{n}^{p - q}} \left(\sum_{k=1}^{n} \int_{0}^{t^{2/(p - q)}} \mathbb{P} \left\{X_{n,k}^{2} > s \right\} \mathrm{d}s \right)^{q/2} \mathrm{d}t \\
&\; \; \leqslant 2^{(q - 2)/2} b_{n}^{-p} \int_{0}^{b_{n}^{p - q}} \left(\sum_{k=1}^{n} t^{2/(p - q)} \mathbb{P} \left\{X_{n,k}^{2} > t^{2/(p - q)} \right\} \right)^{q/2} \mathrm{d}t \\
& \; \; \quad + 2^{(q - 2)/2} b_{n}^{-p} \int_{0}^{b_{n}^{p - q}} \left(\sum_{k=1}^{n} \mathbb{E} \, X_{n,k}^{2} I_{\left\{\lvert X_{n,k} \rvert \leqslant t^{1/(p - q)} \right\}} \right)^{q/2} \mathrm{d}t \\
&\; \; \leqslant 2^{(q - 2)/2} b_{n}^{-p} \int_{0}^{b_{n}^{p - q}} \left(\sum_{k=1}^{n} t^{(2 - p)/(p - q)} \mathbb{E} \, \lvert X_{n,k} \rvert^{p} I_{\left\{\lvert X_{n,k} \rvert > t^{1/(p - q)} \right\}} \right)^{q/2} \mathrm{d}t \\
& \; \; \quad + 2^{(q - 2)/2} b_{n}^{-p} \int_{0}^{b_{n}^{p - q}} \left(\sum_{k=1}^{n} \mathbb{E} \, X_{n,k}^{2} I_{\left\{\lvert X_{n,k} \rvert \leqslant b_{n} \right\}} + \sum_{k=1}^{n} \mathbb{E} \, X_{n,k}^{2} I_{\left\{b_{n} < \lvert X_{n,k} \rvert \leqslant t^{1/(p - q)} \right\}} \right)^{q/2} \mathrm{d}t \\
&\; \; \leqslant 2^{(q - 2)/2} b_{n}^{-p} \left(\sum_{k=1}^{n} \mathbb{E} \, \lvert X_{n,k} \rvert^{p} I_{\left\{\lvert X_{n,k} \rvert > b_{n} \right\}} \right)^{q/2} \int_{0}^{b_{n}^{p - q}} t^{q(2 - p)/(2p - 2q)} \mathrm{d}t \\
&\; \; \quad + 2^{q - 2} \left(\frac{\sum_{k=1}^{n} \mathbb{E} \, X_{n,k}^{2} I_{\left\{\lvert X_{n,k} \rvert \leqslant b_{n} \right\}}}{b_{n}^{2}} \right)^{q/2} + 2^{q - 2} b_{n}^{-p} \int_{0}^{b_{n}^{p - q}} \left(\sum_{k=1}^{n} \mathbb{E} \, X_{n,k}^{2} I_{\left\{b_{n} < \lvert X_{n,k} \rvert \leqslant t^{1/(p - q)} \right\}} \right)^{q/2} \mathrm{d}t \\
&\; \; \leqslant 2^{(q - 2)/2} b_{n}^{-p} \left(\sum_{k=1}^{n} \mathbb{E} \, \lvert X_{n,k} \rvert^{p} I_{\left\{\lvert X_{n,k} \rvert > b_{n} \right\}} \right)^{q/2} \frac{2(q - p) b_{n}^{p - pq/2}}{p(q - 2)} + 2^{q - 2} \left(\frac{\sum_{k=1}^{n} \mathbb{E} \, X_{n,k}^{2} I_{\left\{\lvert X_{n,k} \rvert \leqslant b_{n} \right\}}}{b_{n}^{2}} \right)^{q/2}  \\
&\; \; \quad + 2^{q - 2} b_{n}^{-p} \left(\sum_{k=1}^{n} \mathbb{E} \, \lvert X_{n,k} \rvert^{p} I_{\left\{\lvert X_{n,k} \rvert > b_{n} \right\}} \right)^{q/2} \left[b_{n}^{p - pq/2} + \int_{0}^{b_{n}^{p - q}} t^{q(2 - p)/(2p - 2q)} \, \mathrm{d}t \right] \\
&\; \; = \frac{2^{q/2} (q - p)}{p(q - 2)} \left(\frac{\sum_{k=1}^{n} \mathbb{E} \, \lvert X_{n,k} \rvert^{p} I_{\left\{\lvert X_{n,k} \rvert > b_{n} \right\}}}{b_{n}^{p}} \right)^{q/2}  + 2^{q - 2} \left(\frac{\sum_{k=1}^{n} \mathbb{E} \, X_{n,k}^{2} I_{\left\{\lvert X_{n,k} \rvert \leqslant b_{n} \right\}}}{b_{n}^{2}} \right)^{q/2}  \\
&\; \; \quad + 2^{q - 2} \left(\frac{\sum_{k=1}^{n} \mathbb{E} \, \lvert X_{n,k} \rvert^{p} I_{\left\{\lvert X_{n,k} \rvert > b_{n} \right\}}}{b_{n}^{p}} \right)^{q/2} + \frac{2^{q - 1} (q - p)}{p(q - 2)} \left(\frac{\sum_{k=1}^{n} \mathbb{E} \, \lvert X_{n,k} \rvert^{p} I_{\left\{\lvert X_{n,k} \rvert > b_{n} \right\}}}{b_{n}^{p}} \right)^{q/2} \\
&\; \; = \left[\frac{(2^{q/2} + 2^{q - 1}) (q - p)}{p(q - 2)} + 2^{q - 2} \right] \left(\frac{\sum_{k=1}^{n} \mathbb{E} \, \lvert X_{n,k} \rvert^{p} I_{\left\{\lvert X_{n,k} \rvert > b_{n} \right\}}}{b_{n}^{p}} \right)^{q/2} \\
&\; \; \quad + 2^{q - 2} \left(\frac{\sum_{k=1}^{n} \mathbb{E} \, X_{n,k}^{2} I_{\left\{\lvert X_{n,k} \rvert \leqslant b_{n} \right\}}}{b_{n}^{2}} \right)^{q/2}
\end{align*}
and
\begin{align*}
& b_{n}^{-q} \left(\sum_{k=1}^{n} \int_{0}^{b_{n}^{2}} \mathbb{P} \left\{X_{n,k}^{2} > t \right\} \, \mathrm{d}t \right)^{q/2} \\
& \; \; \leqslant 2^{(q - 2)/2} b_{n}^{-q} \left(\sum_{k=1}^{n} b_{n}^{2} \mathbb{P} \left\{\lvert X_{n,k} \rvert > b_{n} \right\} \right)^{q/2} + 2^{(q - 2)/2} b_{n}^{-q} \left(\sum_{k=1}^{n} \mathbb{E} \, X_{n,k}^{2} I_{\left\{\lvert X_{n,k} \rvert \leqslant b_{n} \right\}} \right)^{q/2} \\
&\; \; \leqslant 2^{(q - 2)/2} b_{n}^{-q} \left(b_{n}^{2 - p} \sum_{k=1}^{n} \mathbb{E} \, \lvert X_{n,k} \rvert^{p} I_{\left\{\lvert X_{n,k} \rvert > b_{n} \right\}} \right)^{q/2} + 2^{(q - 2)/2} \left(\frac{\sum_{k=1}^{n} \mathbb{E} \, X_{n,k}^{2} I_{\left\{\lvert X_{n,k} \rvert \leqslant b_{n} \right\}}}{b_{n}^{2}} \right)^{q/2} \\
&\; \; \leqslant 2^{(q - 2)/2} \left(\frac{\sum_{k=1}^{n} \mathbb{E} \, \lvert X_{n,k} \rvert^{p} I_{\left\{\lvert X_{n,k} \rvert > b_{n} \right\}}}{b_{n}^{p}} \right)^{q/2} + 2^{(q - 2)/2} \left(\frac{\sum_{k=1}^{n} \mathbb{E} \, X_{n,k}^{2} I_{\left\{\lvert X_{n,k} \rvert \leqslant b_{n} \right\}}}{b_{n}^{2}} \right)^{q/2}.
\end{align*}
The remaining assumptions of Theorems~\ref{thr:2.3} and~\ref{thr:2.4} follow from (i), (ii) as in the proof of Corollary~\ref{cor:3.1}. The proof is complete.
\end{proof}

Let $\{\Psi_{n,k}(x), \, 1 \leqslant k \leqslant n, n \geqslant 1 \}$ be an array of functions defined on $[0,\infty)$ satisfying for all $n \geqslant 1$ and every $1 \leqslant k \leqslant n$,
\begin{equation}\label{eq:3.4}
\Psi_{n,k}(0) = 0, \quad 0 < \frac{\Psi_{n,k}(t)}{t^{p}} \uparrow \; \; \text{and} \; \; \frac{\Psi_{n,k}(t)}{t^{q}} \downarrow \; \; \text{as} \; \; 0 < t \uparrow
\end{equation}
for some $1 \leqslant p < q$.

\begin{corollary}\label{cor:3.3}
Let $\{\Psi_{n,k}(x), \, 1 \leqslant k \leqslant n, n \geqslant 1 \}$ be an array of functions defined on $[0,\infty)$ verifying \eqref{eq:3.4} for some $1 \leqslant p < q \leqslant 2$, and $\left\{X_{n,k}, \, 1 \leqslant k \leqslant n, n \geqslant 1 \right\}$ be an array of zero-mean random variables satisfying \eqref{eq:2.1} for such $q$ and some constant sequence $\alpha_{n}$. If $\{b_{n} \}$, $\{c_{n} \}$ are real sequences of positive numbers such that \\

\noindent \textnormal{(1)} $\sum_{n=1}^{\infty} c_{n} \sum_{k=1}^{n} \mathbb{E} \Psi_{n,k}(\lvert X_{n,k} \rvert)/\Psi_{n,k}(b_{n})  < \infty$, \\

\noindent then $\sum_{n=1}^{\infty} c_{n} \mathbb{E} \left(\abs{\sum_{k=1}^{n} X_{n,k}}/b_{n} - \varepsilon \right)_{+}^{p} < \infty$ for all $\varepsilon > 0$.
\end{corollary}

\begin{proof}
From $\Psi_{n,k}(t)/t^{q} \downarrow$ as $0 < t \uparrow$, it follows
\begin{equation}\label{eq:3.5}
\frac{\mathbb{E} \lvert X_{n,k} \rvert^{q} I_{\left\{\lvert X_{n,k} \rvert \leqslant b_{n} \right\}}}{b_{n}^{q}} \leqslant \frac{\mathbb{E} \Psi_{n,k}(\lvert X_{n,k} \rvert I_{\left\{\lvert X_{n,k} \rvert \leqslant b_{n} \right\}})}{\Psi_{n,k}(b_{n})}
\end{equation}
for all $1 \leqslant k \leqslant n$ and $n \geqslant 1$. On the other hand, $\Psi_{n,k}(t)/t^{p} \uparrow$ as $0 < t \uparrow$ entails
\begin{gather}
\frac{\mathbb{E} \lvert X_{n,k} \rvert^{p} I_{\left\{\lvert X_{n,k} \rvert > b_{n} \right\}}}{b_{n}^{p}} \leqslant \frac{\mathbb{E} \Psi_{n,k}(\lvert X_{n,k} \rvert I_{\left\{\lvert X_{n,k} \rvert > b_{n} \right\}})}{\Psi_{n,k}(b_{n})}, \label{eq:3.6} \\
\mathbb{E} \Psi_{n,k}(\lvert X_{n,k} \rvert I_{\left\{\lvert X_{n,k} \rvert \leqslant b_{n} \right\}}) \leqslant \mathbb{E} \Psi_{n,k}(\lvert X_{n,k} \rvert), \label{eq:3.7} \\
\mathbb{E} \Psi_{n,k}(\lvert X_{n,k} \rvert I_{\left\{\lvert X_{n,k} \rvert > b_{n} \right\}}) \leqslant \mathbb{E} \Psi_{n,k}(\lvert X_{n,k} \rvert), \label{eq:3.8}
\end{gather}
for each $1 \leqslant k \leqslant n$ and $n \geqslant 1$. Hence, \eqref{eq:3.5} and \eqref{eq:3.7} yield
\begin{equation*}
\frac{\mathbb{E} \lvert X_{n,k} \rvert^{q} I_{\left\{\lvert X_{n,k} \rvert \leqslant b_{n} \right\}}}{b_{n}^{q}} \leqslant \frac{\mathbb{E} \Psi_{n,k}(\lvert X_{n,k} \rvert)}{\Psi_{n,k}(b_{n})}
\end{equation*}
for any $1 \leqslant k \leqslant n$ and $n \geqslant 1$, which assures assumption (i) of Corollary~\ref{cor:3.1}. Moreover, \eqref{eq:3.6} and \eqref{eq:3.8} imply
\begin{equation}\label{eq:3.9}
\frac{\mathbb{E} \lvert X_{n,k} \rvert^{p} I_{\left\{\lvert X_{n,k} \rvert > b_{n} \right\}}}{b_{n}^{p}} \leqslant \frac{\mathbb{E} \Psi_{n,k}(\lvert X_{n,k} \rvert I_{\left\{\lvert X_{n,k} \rvert > b_{n} \right\}})}{\Psi_{n,k}(b_{n})} \leqslant \frac{\mathbb{E} \Psi_{n,k}(\lvert X_{n,k} \rvert)}{\Psi_{n,k}(b_{n})}
\end{equation}
for every $1 \leqslant k \leqslant n$ and $n \geqslant 1$. Thus, assumption (ii) of Corollary~\ref{cor:3.1} holds via \eqref{eq:3.9} and (1). The proof is complete.
\end{proof}

\begin{corollary}\label{cor:3.4}
Let $\{\Psi_{n,k}(x), \, 1 \leqslant k \leqslant n, n \geqslant 1 \}$ be an array of functions defined on $[0,\infty)$ verifying \eqref{eq:3.4} for some $p \geqslant 1$ and $q > \max\{2,p \}$, and $\left\{X_{n,k}, \, 1 \leqslant k \leqslant n, n \geqslant 1 \right\}$ be an array of zero-mean random variables satisfying \eqref{eq:2.8} for such $q$ and some constant sequences $\beta_{n}$, $\xi_{n}$. If $\{b_{n} \}$, $\{c_{n} \}$ are real sequences of positive numbers such that \\

\noindent \textnormal{(1)} $\sum_{n=1}^{\infty} c_{n} \sum_{k=1}^{n} \mathbb{E} \Psi_{n,k}(\lvert X_{n,k} \rvert)/\Psi_{n,k}(b_{n}) < \infty$, \\

\noindent \textnormal{(2)} $\sum_{n=1}^{\infty} c_{n} \left(\sum_{k=1}^{n} \mathbb{E} \, X_{n,k}^{2} I_{\left\{\lvert X_{n,k} \rvert \leqslant b_{n} \right\}}/b_{n}^{2} \right)^{q/2}  < \infty$, \\

\noindent \textnormal{(3)} $\sum_{n=1}^{\infty} c_{n} \left[\sum_{k=1}^{n} \mathbb{E} \Psi_{n,k}(\lvert X_{n,k} \rvert)/\Psi_{n,k}(b_{n}) \right]^{q/2} < \infty$, \\

\noindent then $\sum_{n=1}^{\infty} c_{n} \mathbb{E} \left(\abs{\sum_{k=1}^{n} X_{n,k}}/b_{n} - \varepsilon \right)_{+}^{p} < \infty$.
\end{corollary}

\begin{proof}
The thesis is a consequence of Corollary~\ref{cor:3.2} by arguing as in the proof of Corollary~\ref{cor:3.3}.
\end{proof}

\begin{remark}
We observe that Theorem 3 of \cite{Wu14} can be obtained via Corollaries~\ref{cor:3.3} and~\ref{cor:3.4} by taking $c_{n} = 1$ for all $n \geqslant 1$ and $\Psi_{n,k}(x)$ not depending on $n,k$ satisfying $\Psi_{n,k}(0) = 0$; indeed, for such sequence $c_{n}$, the assumption $\sum_{n=1}^{\infty} c_{n} \left[\sum_{k=1}^{n} \mathbb{E} \Psi_{n,k}(\lvert X_{n,k} \rvert)/\Psi_{n,k}(b_{n}) \right]^{q/2} < \infty$ can be dropped in Corollary~\ref{cor:3.4}.
\end{remark}

The lemma below gives us a von Bahr-Esseen type inequality for row-wise
pairwise negative quadrant dependent (NQD) triangular arrays (see, for instance, \cite{Lita19}). The proof can be performed as in Theorem 2.1 of \cite{Chen14} by employing the truncation $X_{n,k}' = g_{x^{1/r}}(X_{n,k})$, $1 < r < 2$ and $X_{n,k}'' = X_{n,k} - X_{n,k}'$, being thus omitted.

\begin{lemma}\label{lem:3.1}
Let $1 \leqslant r \leqslant 2$ and $\{X_{n,k}, \, 1 \leqslant k \leqslant n, \, n \geqslant 1 \}$ be a triangular array of zero-mean row-wise pairwise NQD random variables such that $\mathbb{E} \, \lvert X_{n,k} \rvert^{r} < \infty$ for all $n \geqslant 1$ and any $1 \leqslant k \leqslant n$. Then
\begin{equation*}
\mathbb{E} \abs{\sum_{k=1}^{n} X_{n,k}}^{r} \leqslant C(r) \, \sum_{k=1}^{n} \mathbb{E} \, \lvert X_{n,k} \rvert^{r}, \quad n \geqslant 1
\end{equation*}
where $C(r) > 0$ depends only on $r$.
\end{lemma}

\begin{remark}
It is worthy to note that using Lemma~\ref{lem:3.1} in Theorems~\ref{thr:2.1} and~\ref{thr:2.2}, we can extend Theorem 3.7 of \cite{Chen14} to sequences $\{X_{n}, \, n \geqslant 1 \}$ of pairwise NQD and identically distributed random variables, by admitting $X_{n,k} = X_{k}$, $p = r$, $c_{n} = n^{t - 2}$ and $b_{n} = n^{1/\rho}$ $(0 < \rho < 2 )$ with $1 \leqslant r \leqslant 2$, $t \geqslant 1$, and $t \rho < 2$.
\end{remark}

\begin{remark}
Let us point out that all statements presented throughout can be properly extended without effort to general arrays $\{X_{n,j}, \, 1 \leqslant j \leqslant k_{n}, \, n \geqslant 1 \}$, where $\{k_{n} \}$ is a sequence of positive integers such that $k_{n} \rightarrow \infty$ as $n \rightarrow \infty$. Considering also an array $\{\Psi_{n,j}(x), \, 1 \leqslant j \leqslant k_{n}, n \geqslant 1 \}$ of functions defined on $[0,\infty)$ satisfying \eqref{eq:3.4} for every $1 \leqslant j \leqslant k_{n}$ and all $n \geqslant 1$, we conclude from Lemma~\ref{lem:3.1} that our Corollary~\ref{cor:3.3} extends Theorem 1.1 of \cite{Yi18}.
\end{remark}

\begin{corollary}
Let $1 < r < 2$ and $\{X_{n,k}, \, 1 \leqslant k \leqslant n, \, n \geqslant 1 \}$ be a triangular array of row-wise pairwise NQD random variables such that $\mathbb{E} \, \lvert X_{n,k} \rvert^{r} < \infty$ for all $n \geqslant 1$ and any $1 \leqslant k \leqslant n$. If $1 \leqslant p < r$ and $\{b_{n} \}$ is a real sequence of positive constants satisfying, \\

\noindent \textnormal{(1)} $\sum_{n=1}^{\infty} \sum_{k=1}^{n} b_{n}^{-r} \int_{0}^{b_{n}^{r}} \mathbb{P} \left\{\lvert X_{n,k} \rvert^{r} > t \right\} \, \mathrm{d}t < \infty$, \\

\noindent \textnormal{(2)} $\sum_{n=1}^{\infty} \sum_{k=1}^{n} \mathbb{E} \, \lvert X_{n,k} \rvert I_{\left\{\lvert X_{n,k} \rvert > b_{n} \right\}}/b_{n} < \infty$, \\

\noindent \textnormal{(3)} $\sum_{n=1}^{\infty} \sum_{k=1}^{n} b_{n}^{-p} \int_{b_{n}^{p}}^{\infty} \mathbb{P} \left\{\lvert X_{n,k} \rvert^{p} > t \right\} \, \mathrm{d}t < \infty$, \\

\noindent then $\sum_{n=1}^{\infty} \mathbb{E} \left[\abs{\sum_{k=1}^{n} (X_{n,k} - \mathbb{E} \, X_{n,k})}/b_{n} - \varepsilon \right]_{+}^{p} < \infty$ for all $\varepsilon > 0$.
\end{corollary}

\begin{proof}
From previous Lemma~\ref{lem:3.1} we obtain \eqref{eq:2.1} with $q = r$ and $\alpha_{n} = C(r)$. The thesis follows from Theorems~\ref{thr:2.1} and~\ref{thr:2.2} by taking $c_{n} = 1$ for all $n \geqslant 1$.
\end{proof}

\section{Final comments}

In 1947, Hsu and Robbins \cite{Hsu47} introduced the concept of complete convergence (see also \cite{Gut94} for a survey). By taking $c_{n} = 1$ for all $n \geqslant 1$ in \eqref{eq:1.3}, we obtain that $\abs{\sum_{k=1}^{n} (X_{n,k} - \mathbb{E} \, X_{n,k})}/b_{n}$ converges completely to zero: indeed, setting $A_{n}(\varepsilon) := \left\{\omega\colon \abs{\sum_{k=1}^{n} (X_{n,k} - \mathbb{E} \, X_{n,k})}/b_{n} - \varepsilon \geqslant 0 \right\}$, we have, for each $\delta > 0$,
\begin{equation*}
\begin{split}
\mathbb{E} \left[\frac{\abs{\sum_{k=1}^{n} (X_{n,k} - \mathbb{E} \, X_{n,k})}}{b_{n}} - \varepsilon \right]_{+}^{p} &\geqslant \delta \, \mathbb{P} \left\{\left[\frac{\abs{\sum_{k=1}^{n} (X_{n,k} - \mathbb{E} \, X_{n,k})}}{b_{n}} - \varepsilon \right]_{+}^{p} > \delta \right\} \\
&= \delta \, \mathbb{P} \left[\left\{\left[\frac{\abs{\sum_{k=1}^{n} (X_{n,k} - \mathbb{E} \, X_{n,k})}}{b_{n}} - \varepsilon \right]_{+}^{p} > \delta \right\} \cap A_{n}(\varepsilon) \right] \\
&\quad + \delta \, \mathbb{P} \left[\left\{\left[\frac{\abs{\sum_{k=1}^{n} (X_{n,k} - \mathbb{E} \, X_{n,k})}}{b_{n}} - \varepsilon \right]_{+}^{p} > \delta \right\} \cap A_{n}(\varepsilon)^{\complement} \right] \\
&= \delta \, \mathbb{P} \left\{\frac{\abs{\sum_{k=1}^{n} (X_{n,k} - \mathbb{E} \, X_{n,k})}}{b_{n}} > \varepsilon + \delta^{1/p} \right\}
\end{split} 
\quad (p \geqslant 1)
\end{equation*}
Furthermore, by using the elementary inequality $\lvert x \rvert^{p} \leqslant \max(1,2^{p-1}) \left[(\lvert x \rvert - \varepsilon)_{+}^{p} + \varepsilon^{p} \right]$ for every real number $x$ and any $p,\varepsilon > 0$, it follows
\begin{equation*}
\mathbb{E} \abs{\frac{\sum_{k=1}^{n} (X_{n,k} - \mathbb{E} \, X_{n,k})}{b_{n}}}^{p} \leqslant 2^{p-1} \left\{\mathbb{E} \left[\frac{\abs{\sum_{k=1}^{n} (X_{n,k} - \mathbb{E} \, X_{n,k})}}{b_{n}} - \varepsilon \right]_{+}^{p} + \varepsilon^{p} \right\} \qquad (p \geqslant 1)
\end{equation*}
so that, our statements also guarantee the convergence in mean of order $p$ (to zero) for triangular arrays of random variables under the considered assumptions.

\section*{Acknowledgements}

This work is a contribution to the Project UIDB/04035/2020, funded by FCT - Funda\c{c}\~{a}o para a Ci\^{e}ncia e a Tecnologia, Portugal.

\end{document}